\documentclass[11pt]{elsarticle}
\usepackage{color}
\usepackage{graphicx,multirow}
\usepackage{amssymb,mathtools}
\usepackage{epstopdf}
\usepackage{bm}
\usepackage{float}
\usepackage{subcaption}
\usepackage{mathtools}
\usepackage[margin=1in,papersize={8.5in,11in}]{geometry}
\usepackage{times}
\usepackage{amsthm}

\newtheorem{lemma}{\bf Lemma}[section]
\newtheorem{theorem}{\bf Theorem}[section]

\def\vs{\vspace{0.2cm}}

\journal{ArXiv}

\renewenvironment{proof}{{\noindent \em Proof.}}{\begin{flushright}$\square$\end{flushright}}
\theoremstyle{definition}

\usepackage[colorlinks=true,linkcolor=blue,urlcolor=blue,bookmarksopen=true]{hyperref}

\begin{document}
\begin{frontmatter}

\title{Stability analysis of hierarchical tensor methods 
for time-dependent PDEs}

\author[ucsc]{Abram Rodgers}
\ead{akrodger@ucsc.edu}
\author[ucsc]{Daniele Venturi\corref{correspondingAuthor}}
\address[ucsc]{Department of Applied Mathematics\\
University of California Santa Cruz\\ Santa Cruz, CA 95064}
\cortext[correspondingAuthor]{Corresponding author}
\ead{venturi@ucsc.edu}

\begin{abstract} 
In this paper we address the question of whether 
it is possible to integrate time-dependent 
high-dimensional PDEs with hierarchical tensor 
methods and explicit time stepping schemes. 
To this end, we develop sufficient conditions 
for stability and convergence of tensor solutions 
evolving on tensor manifolds with constant rank.
We also argue that the applicability of PDE solvers with 
explicit time-stepping may be limited by time-step restriction 
dependent on the dimension of the problem. 
Numerical applications are presented and discussed 
for variable coefficients linear hyperbolic and parabolic PDEs. 
\end{abstract}

\end{frontmatter}

\section{Introduction}
\label{sec:intro}

Computing the solution of high-dimensional linear PDEs 
has become central to many new areas of application such 
as random media \cite{Torquato}, 
optimal transport \cite{Villani}, 
random dynamical systems \cite{Venturi_MZ,Venturi_PRS}, 
mean field games \cite{Carmona}, 
machine learning, and functional-differential 
equations \cite{venturi2018}.
In an abstract setting these problems involve the 
computation of a function $u(t,{\bf x})$ governed 
by an autonomous evolution equation
\begin{align}
  \frac{\partial u}{\partial t} = {\cal L} u ,   \qquad 
  u(0,{\bf x}) = u_0({\bf x}),\label{nonlinear-ibvp} 
\end{align}
where $u: [0,T]\times \Omega \mapsto\mathbb{R}$ 
is a $d$-dimensional (time-dependent) scalar field defined in 
the spacial domain $\Omega\subseteq \mathbb{R}^d$ ($d\geq 2$), 
and $\cal L$ is a linear operator which may 
may be dependent the spacial variables, and may 
incorporate boundary conditions.
Equation \eqref{nonlinear-ibvp} is first approximated with 
respect to the space variables, e.g., by finite differences 
\cite{Strikwerda}, pseudo-spectral methods \cite{spectral}, 
or Galerkin methods. 
To this end, let us assume for simplicity that $\Omega$ is 
a $d$-dimensional box, i.e., $\Omega=[0,1]^d$.
This allows us to transform \eqref{nonlinear-ibvp} into the system of 
ordinary differential equations 
\begin{equation}
\label{mol-ode}
\frac{d{\bf u}}{dt} = {\bf G}{\bf u}, \qquad {\bf u}(0)={\bf u}_0,
\end{equation}
where ${\bf u}:[0,T]\rightarrow
{\mathbb R}^{n_1\times n_2 \times \cdots \times n_d}$ 
is multi-dimensional array of real numbers (the solution tensor), 
$\bf G$ is a finite dimensional linear operator 
(the discrete form of $\mathcal{L}$). 
The structure of $\bf G$ depends on the 
spacial discretization of $u(t,{\bf x})$, as 
well as on the tensor format utilized for $\bf u$. 
As is well known, any straightforward spacial 
discretization of \eqref{nonlinear-ibvp}  
inevitably leads to the so-called curse of 
dimensionality \cite{Bellman1957}. 
For example, approximating $u(t,{\bf x})$ using a 
finite-dimensional Galerkin basis with $N$ degrees of freedom in 
each spacial variable yields a total number of $N^d$ degrees 
of freedom. 
To address the exponential growth of such degrees of 
freedom, the computational cost and the storage 
requirement, techniques such as sparse collocation 
\cite{Bungartz,Chkifa,Barthelmann,Foo1,Akil}, 
high-dimensional model representations (HDMR) 
\cite{Li1,CaoCG09,Baldeaux}, deep learning 
\cite{Beck2019,Raissi,Raissi1,Zhu2019} 
and hierarchical tensor methods 
\cite{venturi2018,khoromskij,Bachmayr,
parr_tensor,Hackbusch_book,ChinestaBook,Kolda} were 
recently proposed, with the most efficient one 
being problem-specific.
For instance, if a hierarchical Tucker 
(HT) tensor format is utilized, 
then the computational complexity of approximating 
$u(t,{\bf x})$ roughly scales as $\mathcal{O}(d \log N)$ 
instead of  $\mathcal{O}(N^d)$ (tensor product discretization) 
\cite{grasedyck2010hierarchical,lathauwer2000,
grasedyck2018distributed,kressner2014algorithm}.
Combining the spacial discretization \eqref{mol-ode} with 
a discrete ODE formula for the time-stepping, such an 
Adams-Bashforth formula, yields a decoupling of space 
and time which is known as the method of lines. For instance, 
if we discretize \eqref{mol-ode} in time with the 
two-step Adams-Bashforth formula, we obtain  
 \begin{equation}
{\bf u}^{k+1} = {\bf u}^{k}+\frac{\Delta t}{2} {\bf G} \left(3{\bf u}^{k}-{\bf u}^{k-1}\right).
\label{AB2}
\end{equation}
Of particular interest are low-rank hierarchical tensor 
approximations of the solution to \eqref{mol-ode}. 
Such approximations allow us to significantly reduce 
the number of degrees of freedom in the representation  
of the solution tensor ${\bf u}(t)$, while maintaining accuracy.
Low-rank tensor approximations of \eqref{mol-ode} can 
be constructed by using, e.g., rank-constrained temporal 
integration \cite{Lubich2018,hierar,Alec2019} on a  
(smooth) tensor manifold with constant rank \cite{uschmajew2013geometry}.
Alternatively, one can utilize the fully discrete 
scheme \eqref{AB2} followed by a rank-reduction 
operation. To this end, suppose we are given a 
low-rank representation of ${\bf u}^{k}$ and ${\bf u}^{k-1}$. 
Computing ${\bf u}^{k+1}$ based on the scheme \eqref{AB2}
involves addition of tensors, and the application of a 
linear operator. All these operations increase the tensor rank
of the solution, i.e., the storage requirements.
To avoid an undesirable growth of the tensor rank in time, 
we need to {truncate} ${\bf u}^{k}$ 
back to a tensor manifold with constant rank. 
This operation is essentially a nonlinear projection 
which can be computed, e.g., by a sequence of 
matricizations followed by high-order 
singular value decomposition (HOSVD)
\cite{grasedyck2010hierarchical,grasedyck2018distributed,
kressner2014algorithm}, or by optimization 
\cite{Vandereycken2013,da2015optimization,Smith1994,
Kolda,parr_tensor,Silva,Rohwedder,Karlsson}.

The main objective of this paper is to study the 
effects of truncation onto a low-rank tensor 
manifold on the numerical stability of explicit 
linear multistep schemes such as \eqref{AB2}.
To this end, we develop a thorough analysis based 
on a rigorous operator framework, which allows us to 
determine whether rank-constrained LMM integrators
are stable or not. We also argue that the applicability of PDE solvers 
with explicit time-stepping may be limited by time-step 
restriction dependent on the dimension of the spacial variable.

This paper is organized as follows. In section \ref{subsec:explicit-time}
we briefly review stability of linear multistep methods (LMM)  
to solve the ODE \eqref{mol-ode}. To this end, we follow the 
excellent analysis of Reddy and Trefethen \cite{Reddy1990}.
In section \ref{subsec:trunc-explicit-time} we discuss tensor 
rank-reduction methods in linear multistep schemes, and study 
the stability of the corresponding algorithms. 
 In section \ref{sec:stiffness} we
argue that the applicability of PDE solvers with explicit time-stepping 
may be limited by time-step restrictions dependent on the 
dimension of the problem. 
Numerical examples demonstrating the theoretical claims 
are presented and discussed in section \ref{sec:numeric-eg}. 
Finally, the main findings are summarized in 
section \ref{sec:summary}. We also include two brief
appendices where we review classical tensor algebra and the 
hierarchical Tucker tensor format.


\section{Explicit linear multistep methods}
\label{subsec:explicit-time}
The numerical solution to semi-discrete form \eqref{mol-ode} can be computed 
with the use of any ODE solver. In preparation for an analysis
of rank-truncated time stepping algorithms involving hierarchical tensors, 
we will follow the stability analysis of Reddy and Trefethen \cite{Reddy1990}. 
Their analysis follows an explicit $s$-step linear multistep 
method (LMM) of the form
\begin{align}
{{\bf u}}^{k+s} + \sum_{j=0}^{s-1}a_j{{\bf u}}^{k+j} =
\Delta t\sum_{j=0}^{s-1}b_j {\bf G}{{\bf u}}^{k+j},
\label{LMM}
\end{align}
where ${\bf u}^{k}\approx{{\bf u}}(k\Delta t)$ approximates the solution 
to \eqref{mol-ode} at time  $t_k=k\Delta t$ ($k=0,1,\ldots$). 
Upon definition of  
\begin{align}
{\bf v}^{k+1} 
\begin{bmatrix}
{\bf u}^{k+s}\\ {\bf u}^{k+s-1}\\ \vdots \\{\bf u}^{k+1}
\end{bmatrix}
\quad \text{and}\quad
{\bf L}_{\Delta t} = 
\begin{bmatrix}
 b_{s-1}\Delta t {\bf G}-a_{s-1}{\bf I} &
\cdots &
b_{1}\Delta t  {\bf G}-a_1{\bf I} &
b_{0}\Delta t {\bf G}-a_0{\bf I}\\
{\bf I} &  \cdots & {\bf 0}  &{\bf 0} \\
 \vdots & \ddots & \vdots & \vdots \\
{\bf 0}&\cdots & {\bf I} & {\bf 0}
\end{bmatrix}
\label{lin-time-step-op}
\end{align}
we can write the LMM \eqref{mol-ode} 
in a compact form as  
\begin{align}
 {\bf v}^{k+1} = {\bf L}_{\Delta t} {\bf v}^{k},
 \label{LMMrecurrence}
\end{align}
where ${\bf 0}$ is the zero operator and 
${\bf I}$ is the identity operator.
Note that this is not expressed as matrix-vector multiplication 
at each index, but instead is now block application of linear 
operators for different time indexes. 

\subsection{Stability analysis of explicit LMM}
The semi-discrete method \eqref{mol-ode}  
is said to be space-stable (or stable) if for some $T > 0$ 
the solution is bounded for $0 < t < T$, for arbitrary 
initial conditions in some Banach space, 
as the number of degrees of freedom (e.g., mesh points or modes) 
increases. 
In this section we briefly review the definition of stability for 
explicit linear multi-step discretizations. The operator 
${\bf L}_{\Delta t}$ in \eqref{lin-time-step-op} 
defines a stable method under norm $||\cdot ||$  if
\begin{equation}
\left\|
\left (
{\bf L}_{\Delta t}
\right )^n {\bf v}^0 - 
\left (
{\bf L}_{\Delta t}
\right )^n {\bf w}^0
 \right\|
\leq
C_T \left\|
{\bf v}^0 - {\bf w}^0
\right\|
\end{equation}
for any pair of initial states ${\bf v}^0, {\bf w}^0$. 
The real number $C_T\geq 0$ is the Lipschitz
constant of the scheme for a fixed integration 
time $T$. $C_T$ must be independent of
${\bf v}^0$, ${\bf w}^0$, the spacial resolution in $\bf G$, 
as well as $\Delta t$ and $n$ if $n\Delta t \leq T$. 
Reddy and Trefethen \cite{Reddy1990} showed that we may instead
check if the scheme is power-bounded. It is easy to 
see by the linearity of ${\bf L}_{\Delta t}$ that the 
above condition is equivalent to the inequality
\begin{equation}
 \left\|
\left (
{\bf L}_{\Delta t}
\right )^n {\bf z}^0
 \right\|
\leq
C_T  \left\|
{\bf z}^0
 \right\|,
\end{equation}
where ${\bf z}^0 = {\bf v}^0 - {\bf w}^0$ is arbitrary. Clearly, if
${\bf z}^0 = {\bf 0}$ then equality is achieved. So we lose no generality by
imposing ${\bf z}^0\neq {\bf 0}$.  Divide both sides by $||{\bf z}^0||$ to obtain
\begin{equation}
\frac{
 \left\|
\left (
{\bf L}_{\Delta t}
\right )^n {\bf z}^0
 \right\|
}{
  \left\|
{\bf z}^0
 \right\|
}
\leq
C_T.
\end{equation}
It is immediate that the inequality 
holds for all ${\bf z}^0\neq {\bf 0}$ 
if and only if it holds for the maximizing ${\bf z}^0$. 
Therefore, stability for LMM is equivalent to
the time stepping operator being power bounded, i.e., 
\begin{equation}
 \left\|\left (
{\bf L}_{\Delta t}
\right )^n
 \right\|=
\max_{{\bf z}\neq 0}
\frac{
 \left\|
\left (
{\bf L}_{\Delta t}
\right )^n {\bf z}
 \right\|
}{
  \left\|
{\bf z}
 \right\|
}
\leq
C_T.
\end{equation}
If we let $||\cdot|| = ||\cdot||_2$, then the operator 
norm is the spectral radius. Reddy and Trefethen \cite{Reddy1990,Reddy1992}
proved that power bounded property 
for $n\rightarrow \infty$ and fixed $\Delta t$ is equivalent to
a statement about the eigenvalues of ``nearby'' linear operators, 
the so-called $\varepsilon$-eigenvalues. 
Specifically, given any $\varepsilon >0$, there exists $\bf E$ 
with $|| {\bf E}||_2 \leq\varepsilon$ satisfying
\begin{equation}
\label{power-bounded}
\left\|{\bf L}_{\Delta t} + {\bf E} \right\|_2 \leq 1 + C_{T}\varepsilon.
\end{equation}
They also show that if $n\rightarrow \infty$ 
is changed to the condition $n\Delta t\leq T$,
then there is a direct generalization of the Lax stability 
inequality \cite{Lax1956}.
Specifically, given any $\varepsilon >0$, there exists $\bf E$ with
$|| {\bf E}||_2 \leq\varepsilon$ satisfying
\begin{equation}
\label{general-lmm-stable}
\left\| {\bf L}_{\Delta t} + {\bf E} \right\|_2 \leq 1 + K_{T}\varepsilon+Q_{T}\Delta t
\qquad\text{for all $n$ such that} \quad n\Delta t \leq T.
\end{equation}
The inequality above may be interpreted as
``the operator ${\bf L}_{\Delta t}$
can be perturbed in a such a way that its eigenvalues grow
linearly with time step away from a slightly enlarged
unit disk in the complex plane.'' In their paper, the statement
is made in terms of the spectral radius rather than operator norms.

\section{Explicit LMM on low-rank tensor manifolds}
\label{subsec:trunc-explicit-time}

We represent ${\bf u}^{j}$ in \eqref{LMM} or \eqref{LMMrecurrence}
with a hierarchical tensor format corresponding to an arbitrary 
binary tree. Well-known examples of such tensors are 
the hierarchical Tucker (HT) format 
\cite{Hackbusch2009,grasedyck2010hierarchical}, 
and the tensor-train (TT) format \cite{OseledetsTT}.
Most algebraic operations between tensors, 
including tensor addition and the application of a linear 
operator to a tensor, increase the tensor rank. 
Therefore, to avoid an unbounded growth in time 
of tensor rank of the solution obtained by an iterative 
application of \eqref{LMM} or \eqref{LMMrecurrence}, 
we need {\em truncate} ${\bf u}^{k}$ back to a tensor 
manifold of constant rank. This operation is essentially a 
nonlinear projection which can be computed, e.g., by a 
sequence of matricizations followed by hierarchical 
singular value decomposition (SVD)
\cite{grasedyck2010hierarchical,grasedyck2018distributed,kressner2014algorithm}, 
by Riemannian optimization 
\cite{Vandereycken2013,da2015optimization,Smith1994,Kolda,parr_tensor,Silva,Rohwedder,Karlsson}, 
or by rank-constrained temporal integration \cite{Lubich2018,hierar}.
Hereafter we study a simple algorithm for rank-constrained 
temporal integration, where we truncate  ${\bf u}^{k}$ 
with high-order SVD. The method has the form 
\begin{equation}
\label{truncated-lmm}
{{\bf u}}^{k+s} =  {\mathfrak T}_r\left(\sum_{j=0}^{s-1}
\Delta t b_j {\bf G}{\bf u}^{k+j} -
a_j{{\bf u}}^{k+j} \right ),
\end{equation}
where ${\mathfrak T}_r(\cdot)$ is the rank-$r$ 
(nonlinear) truncation operator. 
Equation \eqref{truncated-lmm} gives us the simplest 
truncated tensor method for solving high dimensional 
linear PDE of the form \eqref{nonlinear-ibvp}.
As before, we transform the $s$-term recurrence 
\eqref{truncated-lmm} into a 1-term recurrence as 
\begin{align*}
\begin{bmatrix}
{\bf u}^{k+s}\\ {\bf u}^{k+s-1}\\ \vdots \\{\bf u}^{k+1}
\end{bmatrix}
=
\begin{bmatrix}
\displaystyle{\mathfrak T}_r\left(\sum_{j=0}^{s-1}
\Delta t b_j {\bf G}{\bf u}^{k+j} -
a_j{{\bf u}}^{k+j} \right )
\\ {\bf u}^{k+s-1}\\ \vdots \\{\bf u}^{k+1}
\end{bmatrix}.
\end{align*}
However, in this case we lose the operator multiplication 
property. I.e., we cannot write the scheme as 
${\bf v}^{k+1}=\mathfrak{T}_r\left({\bf L}_{\Delta t}{\bf v}^{k} \right)$ 
unless we truncate all tensors $\{{\bf u}^{k},\ldots,{\bf u}^{k+s-1}\}$ 
to rank $r$. If we do so, then we have
${\mathfrak T}_r({\bf u}^{k+j}) = {\bf u}^{k+j}$ for all $j < s$, 
as well as the most recent time step $j=s$. This yields the scheme 
\begin{align*}
\begin{bmatrix}
{\bf u}^{k+s}\\ {\bf u}^{k+s-1}\\ \vdots \\{\bf u}^{k+1}
\end{bmatrix}
=
\begin{bmatrix}
\displaystyle{\mathfrak T}_r\left(\sum_{j=0}^{s-1}
\Delta t b_j {\bf G}{\bf u}^{k+j} -
a_j{{\bf u}}^{k+j} \right )
\\{\mathfrak T}_r( {\bf u}^{k+s-1})\\ \vdots \\{\mathfrak T}_r({\bf u}^{k+1})
\end{bmatrix},
\end{align*}
which we denote more concisely as 
\begin{align}
 {\bf v}^{k+1} = {\mathfrak T}_r({\bf L}_{\Delta t} {\bf v}^{k}).
 \label{truncatedLMMscheme}
\end{align}
The truncation operator removes linearity from this method, since for any 
two HT tensors ${\bf x}$ and ${\bf y}$ we have 
\begin{equation}
  {\mathfrak T}_r({\bf x}) + {\mathfrak T}_r({\bf y}) \neq
  {\mathfrak T}_r({\bf x}+{\bf y}).
\end{equation}
This can easily be seen by looking at an example where
${\bf x}=(1,0)\otimes(1,0)$, ${\bf y} = (0,\frac{1}{2})\otimes(0,1)$, and
${\mathfrak T}_r$ is the rank 1 matrix truncation operator computed using
the SVD. Since $\bf x$ and
$\bf y$ both have rank 1, the truncation operator does nothing and so the
left side is rank 2. The right side is rank 1 by definition, therefore not
equal. However, the following important property holds true
\begin{equation}
{\mathfrak T}_r(\alpha {\bf x}) = \alpha {\mathfrak T}_r({\bf x})
\label{scalability}
\end{equation}
for a scalar $\alpha$. This can be shown by using the recursive definition
of the high-order SVD \cite{grasedyck2010hierarchical}. 
We'll call property \eqref{scalability} {\em scalability} of the 
rank-truncation operation.

\paragraph{Remark} It can be verified easily that scalable
functions are a vector space and a subspace of the set of
all functions ${\mathbb R}^N\mapsto{\mathbb R}^M$. This
will be used to define a norm of the space of these scalable functions.

\subsection{Stability analysis of LMM on low-rank tensor manifolds}
\label{sec:stability}

We wish to establish a relationship between the LMM time 
stepping operator ${\bf L}_{\Delta t}$ and its rank-truncated 
version ${\mathfrak T}_r({\bf L}_{\Delta t} (\cdot))$. Such 
relationship is nontrivial, as highlighted by the following example.
Suppose ${\bf L}_{\Delta t}$ is a linear contraction in a 
specified norm, i.e., a linear map
whose Lipschitz constant is $C\in[0,1)$. Then
${\mathfrak T}_r({\bf L}_{} (\cdot))$ need not be 
a contraction in that same norm.
At first, this appears to be telling us to lose hope on
maintaining stability after truncation. However, there are
a few remarkable facts about the nature of the truncation operator 
${\mathfrak T}_r$ which suggest that such operator is either 
neutral or can enhance stability. 
Firstly, if the numerical solution is always below a known 
rank\footnote{The numerical solution to a
constant coefficient advection equation with separable initial condition is 
always rank one \cite{Alec2019}.},
then the rank-truncation operator is an identity operator on that
set of known rank. In this case, the classical linear stability analysis
can be applied. Secondly, we will see that the relationship
between ${\bf L}_{\Delta t}$ and ${\mathfrak T}_r({\bf L}_{\Delta t} (\cdot))$ 
defines a type of stability. 

We begin our analysis by recalling the definition
of seminorm of a nonlinear  function 
${\bf T}:{\mathbb R}^N\rightarrow{\mathbb R}^M$. The seminorm 
of ${\bf T}$ is the same as the norm of a linear map, but since $\bf T$
need not be continuous, we replace $\max$ with $\sup$.
\begin{align}
    \left\|{\bf T}\right\| = \sup_{{\bf z}\neq{\bf 0}}
    \frac{\left\|{\bf T}({\bf z})\right\|}{\left\|{\bf z}\right\|}.
    \label{seminorm}
\end{align}
It can be verified that the above definition obeys
triangle inequality and absolute scalability. Moreover, 
for any ${\bf w}\neq {\bf 0}$ we have
\begin{align*}
     \frac{\left\|{\bf T}({\bf w})\right\|}{\left\|{\bf w}\right\|} 
     \leq \sup_{{\bf z}\neq{\bf 0}}
    \frac{\left\|{\bf T}({\bf z})\right\|}{\left\|{\bf z}\right\|}
\end{align*}
by definition of supremum. Multiplying by $\left\|{\bf w}\right\|$ 
the denominator yields an inequality that is very similar to  
Cauchy-Schwartz
\begin{equation}
\label{almost-cauchy-shwarz}
    {\left\|{\bf T}({\bf w})\right\|}
     \leq {\left\|{\bf w}\right\|}  \sup_{{\bf z}\neq{\bf 0}}
    \frac{\left\|{\bf T}({\bf z})\right\|}{\left\|{\bf z}\right\|}
    = \left\|{\bf w}\right\|\left\|{\bf T}\right\|.
\end{equation}
%
Now, suppose ${\bf T}$ satisfies the scalability property, i.e., 
${\bf T}(\alpha {\bf z}) = \alpha {\bf T}({\bf z})$, as with
the rank truncation operator. Then we can pass the 
norm of $\bf z$ into the numerator.
\begin{align*}
    \left\|{\bf T}\right\| = \sup_{{\bf z}\neq{\bf 0}}
    \frac{\left\|{\bf T}({\bf z})\right\|}{\left\|{\bf z}\right\|}=
     \sup_{{\bf z}\neq{\bf 0}}
    \left\|
    {\bf T}\left ( \frac{{\bf z}}{\left\|{\bf z}\right\|}\right )
    \right\|
    = \sup_{\left\|{\bf u}\right\|={ 1}}
    \left\|
    {\bf T}\left ( {\bf u} \right )
    \right\|.
\end{align*}
In other words, for scalable functions, we can take 
maximization over the unit sphere in a given norm. 
Additionally, the norm here is arbitrary. 
Now we show that the operator semi-norm defines
a norm on the vector space of scalable functions.
Essentially, we need to show that for scalable $\bf T$,
$||{\bf T}|| = 0$ implies $\bf T$ is zero everywhere.
To this end, a proof by contradiction is sufficient. 
Suppose ${\bf T}({\bf w})\neq {\bf 0}$. Then $||{\bf T}({\bf w})||> 0$.
Since for any $\bf v$, we have 
${\bf T}({\bf 0}) ={\bf T}(0{\bf v})=0 {\bf T}({\bf v})={\bf 0}$,
we must have ${\bf w}\neq {\bf 0}$. So the ratio of 
$||{\bf T}({\bf w})||$ and $||{\bf w}||$
is positive. This implies that 
\begin{align*}
    0 < \frac{\left\|{\bf T}({\bf w})\right\|}{\left\|{\bf w}\right\|} =
    \left\|
    {\bf T}\left ( \frac{{\bf w}}{\left\|{\bf w}\right\|}\right )
    \right\|
    \leq \left\|{\bf T}\right\|=0,
\end{align*}
i.e., $0<0$, a contradiction. Therefore the operator 
semi-norm \eqref{seminorm} induces a norm on the vector 
space of scalable functions. 
In other words, the operator norm is well defined for 
scalable maps. Rather than the language of contractions, 
we now use the operator norm of a truncated multi-step 
time stepping operator to describe behavior 
of iterated application. 

\paragraph{Remark}
Suppose the map $\bf T$ is scalable and Lipschitz. 
Then for any $\bf u\neq 0$
\begin{align*}
     \left\||
    {\bf T}\left ( {\bf u} \right ) 
    - {\bf T}\left ( {\bf 0} \right )
    \right\|
    \leq
    C \left\|{\bf u} - {\bf 0}\right\|\quad \Rightarrow \quad
       \frac{ \left\|
    {\bf T}\left ( {\bf u} \right ) 
    \right\|}{
       \left\|{\bf u}\right\|
    } \leq C \quad \Rightarrow \quad \sup_{\left\|{\bf u}\right\|=1} \left\|{\bf T}({\bf u})\right\|\leq C
\end{align*}
since ${\bf T}\left ( {\bf 0} \right ) = {\bf 0}$ ($\bf T$ is scalable).
This means that the Lipschitz constant of $\bf T$ is an 
upper bound for the operator norm.

\vs
\noindent
Up to this point, the discussion has been developed for arbitrary norms.
To discuss the stability of LMM integrators on low-rank tensor manifolds 
\eqref{truncatedLMMscheme}, we consider
the 2-norm in particular. This is the norm computed
by squaring all entries of a tensor, summing, and then
taking square root.

\paragraph{Remark}
We recall that the singular value decomposition of
a matrix $\bf x$ is differentiable
with respect to $\bf x$ if the singular
values are nonzero and unique. Hence, ${\mathfrak T}_r$ 
is Lipschitz on every compact set containing only 
points at which the SVD is differentiable. 
If the differentiability assumption is not made,
the SVD is not even unique up to ordering of
the singular values.

\vs
\noindent
The following result characterizes the truncation operator 
 ${\mathfrak T}_r$ as a bounded nonlinear projection.
\begin{lemma}
\label{lemmaTr}
 The operator 2-norm of the hierarchical rank-truncation
 operator is 1, i.e., 
\begin{align}
    \sup_{{\bf x}\neq {\bf 0}}
        \frac{\left\|{\mathfrak T}_r({\bf x})\right\|_2}
        {\left\|{\bf x}\right\|_2}=1.
\end{align}
\end{lemma}

\begin{proof}
 The hierarchical truncation can be computed by 
 (see \cite{grasedyck2010hierarchical} and \ref{sec:htformat})
    \begin{align*}
        {\mathfrak T}_r({\bf x}) &=
        \prod_{t\in{\cal T}_d^p}{\bf P}_t
        \cdots
        \prod_{t\in{\cal T}_d^1}{\bf P}_t{\bf x},
    \end{align*}
    where every ${\bf P}_t$ is an orthogonal projection
    formed using $t$-mode matricizations of $\bf x$.
    The particular ${\bf P}_t$ are dependent on a given
    $\bf x$. Recall that this implies the truncation
    operator is  nonlinear, but still scalable. 
    Since all ${\bf P}_t$ are orthogonal projections,
    they all have the property
    \begin{equation}
        ||{\bf P}_t{\bf v}||^2_2 = 
        \langle {\bf P}_t{\bf v},{\bf P}_t{\bf v}\rangle
        = \langle {\bf P}_t^\top {\bf P}_t{\bf v},{\bf v}\rangle
        =\langle {\bf P}_t {\bf P}_t{\bf v},{\bf v}\rangle
        =\langle {\bf P}_t{\bf v},{\bf v}\rangle
        \leq \left\|{\bf P}_t{\bf v}\right\|_2 \left\|{\bf v}\right\|_2.
    \end{equation}
    Dividing by $\left\|{\bf P}_t{\bf v}\right\|_2$, we have
    \begin{equation}
        \frac{\left\|{\bf P}_t{\bf v}\right\|_2}{\left\|{\bf v}\right\|_2}\leq 1,
    \end{equation}
    for arbitrary $\bf v$. Therefore, the operator norm is
    at most 1. Now apply this to the composition of operators
    which defines hierarchical truncation.
    \begin{align*}
        \left\|{\mathfrak T}_r({\bf x})\right\|_2 &=
        \left\|
        \prod_{t\in{\cal T}_d^p}{\bf P}_t
        \cdots
        \prod_{t\in{\cal T}_d^1}{\bf P}_t{\bf x}
          \right\|_2 \\
    &\leq
        \left\|
        \prod_{t\in{\cal T}_d^p}{\bf P}_t
        \cdots
        \prod_{t\in{\cal T}_d^1}{\bf P}_t
          \right\|_2 \left\|{\bf x}\right\|_2\\
    &\leq \prod_{t\in{\cal T}_d^p}
        \left\|{\bf P}_t\right\|_2
        \cdots
         \prod_{t\in{\cal T}_d^1}\left\|{\bf P}_t\right\|_2
         \left\|{\bf x}\right\|_2\\
    &\leq \left\|{\bf x}\right\|_2
    \end{align*}
    Dividing by $\left\|{\bf x}\right\|_2$, we get
    \begin{equation}
        \frac{\left\|{\mathfrak T}_r({\bf x})\right\|_2}{\left\|{\bf x}\right\|_2}\leq 1.
    \end{equation}
    Equality is achieved by noting that
    ${\mathfrak T}_r({\mathfrak T}_r({\bf x}))={\mathfrak T}_r({\bf x})$,
    \begin{align*}
        ||{\mathfrak T}_r({\mathfrak T}_r({\bf x}))||_2
        &=||{\mathfrak T}_r({\bf x})||_2,\\
        \frac{||{\mathfrak T}_r({\mathfrak T}_r({\bf x}))||_2}
        {||{\mathfrak T}_r({\bf x})||_2}&= 1.
    \end{align*}
\end{proof}
%

\noindent
We now have all elements to prove stability of linear multistep 
integrators on low-rank tensor manifolds. 

\begin{theorem} (Stability of LMM on low-rank tensor manifolds)
\label{theorem-trunc-bounded}
Suppose
\begin{equation}
\label{linear-recurrence}
{\bf v}^{k+1} ={\bf L}_{\Delta t} {\bf v}^{k}
\end{equation}
 defines a Lax-stable linear multistep method,  i.e. 
 $\left\|{\bf L}_{\Delta t}\right\|_2 \leq 1+K\Delta t$. 
 Then the rank-truncated scheme 
 \begin{equation}
 {\bf v}^{k+1} = {\mathfrak T}_r({\bf L}_{\Delta t} {\bf v}^{k})
 \label{truncated}
\end{equation}
is stable as long as the LMM scheme \eqref{LMMrecurrence}
is stable. 
\end{theorem}
\begin{proof}
  	We proceed by induction. For $k=1$ the theorem follows
    from inequality \eqref{almost-cauchy-shwarz}. For $k> 1$ we 
    utilize Lemma \ref{lemmaTr}, and write  
\begin{equation}    
     \left\|{\mathfrak T}_r(
     {\bf L}_{\Delta t}{\bf v}^{k})\right\|_2\leq
    \left\|{\bf L}_{\Delta t}\right\|_2^{k}\left\|{\bf v}^0\right\|_2.
    \end{equation}
    Then we recall that stability of a LMM is equivalent to
    \begin{equation}
    \left\|({\bf L}_{\Delta t})^k \right\|_2 \leq C_{T}
    \quad \forall k\Delta t\leq T.
    \end{equation}
    Assuming that the linear recurrence \eqref{linear-recurrence} 
    is Lax-stable, we have
    \begin{equation*}
      \left\|({\bf L}_{\Delta t})^k \right\|_2
      \leq \left\|({\bf L}_{\Delta t}) \right\|_2^k
      \leq(1+K\Delta t)^k
      \leq e^{K\Delta t\cdot k}=e^{KT}=C_T.
    \end{equation*}
    Therefore,
    \begin{align*}
        \left\| {\bf v}^{k+1} \right\|_2 &=
         \left\|{\mathfrak T}_r({\bf L}_{\Delta t} {\bf v}^{k})\right\|_2 \\
         &\leq 
          \left\|{\mathfrak T}_r\right\|_2 
          \left\|{\bf L}_{\Delta t} {\bf v}^{k}\right\|_2 \\
         &\leq \left\|{\bf L}_{\Delta t} {\bf v}^{k}\right\|_2 \\
         &=\left\|{\bf L}_{\Delta t}{\mathfrak T}_r(
            {\bf L}_{\Delta t}{\bf v}^{k-1})\right\|_2\\
         &\leq \left\| {\bf L}_{\Delta t}\right\|_2\left\|{\mathfrak T}_r(
            {\bf L}_{\Delta t}{\bf v}^{k-1})\right\|_2\\
         &\leq \left\| {\bf L}_{\Delta t}\right\|_2^k\left\|{\bf v}^{0}\right\|_2
            \qquad\qquad\qquad\text{(Inductive Hypothesis)}\\
        &\leq C_T \left\|{\bf v}^{0}\right\|_2.
    \end{align*}
By inductive hypothesis, the iteration remains bounded.
\end{proof}

\noindent
Theorem \ref{theorem-trunc-bounded} states that if LMM 
is stable then the rank-truncated LMM is stable. 
This does not exclude the possibility that the truncation 
operator $\mathfrak{T}_r$ can stabilize an unstable LMM
scheme.

\subsection{Consistency and convergence}

The Lax-Richtmyer equivalence theorem states that 
if a method is consistent and stable then the numerical solution 
converges to the solution of the differential equation. 
Clearly, if the hierarchical ranks of ${\bf v}^k$
remain below the truncation rank and the
linear scheme \eqref{linear-recurrence} is 
consistent, then the truncated scheme is 
convergent, i.e., the error goes to zero as 
the number of degrees of freedom (e.g., mesh points or modes) 
increases. To show this it is sufficient to notice that 
if the hierarchical ranks remain below the truncation rank 
for all time then ${\mathfrak T}_r({\bf L}_{\Delta t} {\bf v}^{k})=
{\bf L}_{\Delta t} {\bf v}^{k}$, i.e., the truncation operation is 
essentially the identity. 
More generally, finite-rank tensor schemes can be consistent
if and only if the rank of the analytical solution is finite-rank. This happens, 
for example, in constant coefficient advection or diffusion problems 
with separable initial conditions in periodic domains.
If the analytic solution is of infinite rank then to establish 
consistency one must find a way to raise the hierarchical ranks 
at a rate that depends on the discretization. This requires a 
problem-specific analysis that is beyond the scope of this paper.

\section{Stiffness in high-dimensional PDEs}
\label{sec:stiffness}
In this section, we argue that the applicability of 
PDE solvers with explicit time-stepping such 
as \eqref{LMMrecurrence} may be limited by 
time-step restrictions dependent on dimension 
$d$. Rather than formulating a general theorem on this matter, 
we provide a simple example in which we compute 
the Lipschitz constant associated with a few common linear 
PDE operators. Since the Lipschitz constant corresponds
directly with the time step for many ODE solvers
\cite{hairer1991solving,wanner1996solving}, 
this quantity provides information on what type of scheme one 
can use to integrate the PDE forward in time while maintaining stability.
The relation between the time step and the Lipschitz constant 
is nontrivial, but it often happens that these quantities are 
inversely proportional. This is true, e.g., for implicit
$s$-stages Runge-Kutta methods 
\cite[Theorem 7.2]{hairer1991solving},
where the condition
\begin{equation}
\Delta t < \left(L \max_i \sum_{j=1}^s \left|a_{ij}\right|\right)^{-1}
\label{RK}
\end{equation}
guarantees a unique numerical solution when 
iterating the RK scheme. In \eqref{RK}, 
$a_{ij}$ is the Butcher tableau, and $L > 0$ is the 
Lipschitz constant of velocity vector at the right hand side 
of equation \eqref{mol-ode}.
Since ${\bf G}$ is linear, the Lipschitz constant 
can be stated as 
\begin{equation}
 \left\| {\bf G} {\bf z}\right\| \leq L \left\|{\bf z}\right\| \qquad \forall {{\bf z}\in K},
\end{equation}
where $\left\|\cdot \right\|$ is a suitable norm, and $K$ is the 
phase space domain. If we utilize the 2-norm $\|\cdot \|_2$, then  
$\left\|\bf G\right\|_2$ is the spectral radius of the matrix 
$\bf G$. 
It was shown in \cite{wanner1996solving} via analysis and 
numerical examples that explicit time-stepping schemes 
can detect a stiff problem if the largest eigenvalue of
${\bf G}$ (an approximate local Lipschitz constant) lies 
on the boundary of the region of stability when scaled 
with step size. In a simpler setting, we can perform a Von-Neumann 
stability analysis \cite{Lax1956} (when possible), to show conditional 
or unconditional stability of explicit time-stepping schemes. 
To illustrate these concepts we consider the
following simple initial value problem 
\begin{equation}
    \label{const-coef-diff-eq}
  \frac{\partial u}{\partial t} =
  \sum_{k=1}^{d} c_k\frac{\partial^2 u}{\partial x_k^2},
  \qquad u(0,{\bf x}) = u_0({\bf x})
\end{equation}
in the domain $\Omega=[0,2\pi]^d$, with periodic 
boundary conditions. Equation \eqref{const-coef-diff-eq} 
is a simplified version of the PDEs derived in \cite{venturi2018}. 
In that paper, implicit and explicit numerical schemes for finite-dimensional 
approximations of Functional Differential Equations 
(FDEs) are discussed in great depth. 
A straightforward technique to solving such problems is 
to approximate the solution functional in the span of 
a $d$-dimensional basis. This yields a $d$-dimensional 
linear PDE which needs to be integrated in time. It was 
argued in \cite[\S 7.3.2]{venturi2018}
that explicit time stepping methods to solve such PDE 
need to operate with time steps that are in inverse 
proportionality with a power law of the dimension $d$. 
Hence, the larger the dimension the smaller the time 
step. 
%
Implicit temporal integrators can mitigate this 
problem, but they require the development of 
linear solvers on tensor manifolds with constant rank. 
This can be achieved, e.g., by utilizing Riemannian 
optimization algorithms
\cite{Vandereycken2013,da2015optimization,Smith1994,heidel2018riemannian}, 
or alternating least squares \cite{Etter,Kolda,Rohwedder,parr_tensor}.

Let us discretize the spacial derivatives in \eqref{const-coef-diff-eq}
with second-order centered finite differences on a tensor product 
evenly-spaced grid in each variable. This yields the semi-discrete form 
\begin{equation}
\label{const-coef-advect-mol}
\frac{d{\bf u}}{dt}=
\sum_{k=1}^{d}\frac{c_k}{\Delta x_k^2}\left (
{\bf u}(t,[j_1,\dots,j_{k + 1},\dots,j_d])-
2{\bf u}(t,[j_1,\dots,j_k,\dots,j_d])+
{\bf u}(t,[j_1,\dots,j_{k - 1},\dots,j_d])
\right),
\end{equation}
where $[j_1,\dots,,\dots,j_d]$ labels an entry of the 
tensor $\bf u$.
Following the classical Von-Neumann stability 
analysis \cite{Lax1956,Strikwerda,randall2007leveque}, 
we compute the discrete Fourier transform of the solution tensor 
\begin{equation}
\label{multi-var-dft}
{\bf u}(t,[j_1,j_2,\dots,j_d])=
\sum_{q_1 = 0}^{n_1-1}\sum_{q_2 = 0}^{n_2-1}\cdots
\sum_{q_d = 0}^{n_d-1}{\widehat{\bf u}}(t,[q_1,q_2,\dots,q_d])\exp\left [
{i\left (\sum_{k=1}^d j_k2\pi q_k\Delta x_k\right )}\right ].
\end{equation}
A substitution of \eqref{multi-var-dft} into 
\eqref{const-coef-advect-mol} yields 
\begin{align*}
&\sum_{q_1 = 0}^{n_1-1}\sum_{q_2 = 0}^{n_2-1}\cdots
\sum_{q_d = 0}^{n_d-1}\exp\left [
{i\left (\sum_{m=1}^d j_m 2\pi q_m\Delta x_m\right )}\right ]
\frac{d \widehat{\bf u}}{dt} =\\
&-4\sum_{q_1 = 0}^{n_1-1}\sum_{q_2 = 0}^{n_2-1}\cdots
\sum_{q_d = 0}^{n_d-1}\exp\left [i
{\left (\sum_{m=1}^d j_m 2\pi q_m\Delta x_m\right )}\right ]
\widehat{\bf u}
\sum_{k=1}^{d}  \frac{c_k\sin(2\pi q_k\Delta x_k)^2}{\Delta x_k^2}.
\end{align*}
We recognize that each of the terms in the spacial 
sum are orthogonal with respect to a standard Hermitian 
inner product. Hence, we can just compare the Fourier 
amplitudes term-by-term. This yields the complex-valued 
linear ODE
\begin{equation}
\label{diffusion-mol-dft}
\frac{d \widehat{\bf u}}{dt}=- 4 \widehat{\bf u}
\sum_{k=1}^{d}  \frac{c_k}{\Delta x_k^2}
\sin\left(2\pi q_k\Delta x_k\right)^2.
\end{equation}
The solution to the ODE decays in time, and therefore it shares the same 
qualitative behavior with the original initial value problem 
\eqref{const-coef-diff-eq}. By approximating the temporal derivative with 
the explicit one-step LMM method (Euler forward), and assuming 
an evenly spaced grid with the same spacing in each variable we obtain 
\begin{equation}
\widehat{\bf u}^{n+1} = \widehat{\bf u}^{n}\left[1 - 4 \frac{\Delta t}{\Delta x^2}
\sum_{k=1}^{d}  c_k
\left (\sin(2\pi q_k\Delta x_k)\right)^2\right].
\label{AB1}
\end{equation}
The amplification factors \cite{Lax1956} of this ODE are real and given by
\begin{equation}
g= \left[1 - 4 \frac{\Delta t}{\Delta x^2}
\sum_{k=1}^{d}  \frac{c_k}{\Delta x^2}
\sin\left(2\pi q_k\Delta x_k\right)^2\right].
\end{equation}
A necessary and sufficient condition for Lax-Richtmyer stability 
of \eqref{AB1} is $\left| g\right| \leq 1$ (see \cite{Lax1956}). 
This condition is equivalent to 
\begin{equation}
\Delta t \leq \frac{\Delta x^2}{\displaystyle 2d\max_j\{c_j\}}.
\end{equation}
Hence, Von-Neumann stability analysis suggests that the simple 
one-step method 
\eqref{AB1} is conditionally stable, with stiffness that increases with 
with the dimension $d$. In other words, the larger $d$ the 
smaller $\Delta t$.

\section{Numerical examples}
\label{sec:numeric-eg}
In this section we provide demonstrative examples of 
truncated linear multistep tensor methods applied 
to variable-coefficient advection-diffusion PDEs 
of the form
\begin{align}
  \frac{\partial}{\partial t} u(t,{\bf x}) = 
  -\sum_{k=1}^d \frac{\partial}{\partial x_k} 
        \big (f_k({\bf x}) u(t,{\bf x})\big )
	&+\sum_{k=1}^d\sum_{q=1}^d\frac{\partial^2}{\partial x_k \partial x_q}
		(\Gamma_{kq}({\bf x}) u(t,{\bf x})), 
\label{fp-eq}
\end{align}
where ${\bf f}({\bf x}) = [f_1({\bf x})\ f_2({\bf x})\ \dots \ f_d ({\bf x})]^T$ is the drift vector field and 
${\bf \Gamma}({\bf x}) = [\Gamma_{kq}({\bf x})]$ is the symmetric 
positive-definite diffusion matrix.
As is well-known, the PDE \eqref{fp-eq} governs the evolution 
of the probability density function corresponding to an 
ODE driven by multiplicative white noise \cite{risken1989}.
We approximate the solution of \eqref{fp-eq} in the 
spacial domain $\Omega=[0,2\pi]^d$ with 
periodic boundary conditions. 
In particular, we look at the growth of matrix rank in 
the case of a 2D hyperbolic PDE. Additionally, we provide 
examples of Theorem \ref{theorem-trunc-bounded} in the 
case of in higher-dimensional advection-diffusion PDEs. 
The C++/MPI code hierarchical Tucker code we developed 
to study these examples is available at \cite{rodgers2019htuckermpi}.

\subsection{Two-dimensional hyperbolic PDE}
\label{subsec:2d-advect}
Let us consider the two-dimensional hyperbolic PDE 
experiment with is
\begin{align}
  \frac{\partial}{\partial t} u(t,{\bf x}) &= 
  -\frac{\partial}{\partial x_1} 
        \big (\sin({x_2})\cdot u(t,{\bf x})\big )
   -\frac{\partial}{\partial x_2} 
        \big (\cos({x_1})\cdot u(t,{\bf x})\big ).
  \label{2d-var-coef-advect}
\end{align}

We discretize \eqref{2d-var-coef-advect} in space 
on an evenly spaced grid with $n\times n$ points in $[0,2\pi]^2$. 
Specifically we will consider both Fourier spectral methods and 
second-order finite-differences discretization. 
In the two-dimensional setting we consider here, the semi-discrete 
form \eqref{mol-ode} involves two-dimensional arrays, i.e. matrices.
Hence, hierarchical rank is the same as matrix rank in this case,
since the rank of a matrix and its transpose coincide.
Applying the two-step Adams-Bashforth method \eqref{AB2}. 
yields a linear recurrence relation of the 
form \eqref{linear-recurrence},  with
\begin{equation}
{\bf L}_{\Delta t}=
\begin{bmatrix}
\displaystyle -\frac{\Delta t }{2}{\bf G}+ {\bf I} &
\displaystyle  \frac{3}{2}\Delta t {\bf G}\\
{\bf I}& {\bf 0}
\end{bmatrix},
\label{L2d}
\end{equation}
and 
\begin{equation}
{\bf G}=-
\left( {\bf D}\otimes{\bf I}\right) {\text{diag}}[\sin({\bf x}_2)]+
\left( {\bf I} \otimes{\bf D}\right){\text{diag}}[\cos({\bf x}_1)].
\end{equation}
Here, $\sin({\bf x_2})$ and $\cos({\bf x_1})$ are vectorizations of 
$\sin(x_2)$ and $\cos(x_1)$ evaluated the 2D evenly-spaced spacial 
grid, and $\bm D$ is the first-order (one-dimensional) 
differentiation matrix.

In figure \ref{fig:2d-op-norm} 
we plot the typical behavior of the operator norm 
$\left\|{\bf L}^k_{\Delta t}\right\|$ versus $k$ 
for two conditionally stable schemes, namely the 
Fourier pseudo-spectral and the second-order 
centered finite-difference schemes on a $n\times n$ grid, 
with $n=4,8,16,32,64$. It is seen that 
$\left\|{\bf L}^k_{\Delta t}\right\|$ grows as $k^2$, 
in the case of the Fourier pseudo-spectral method, making it 
considerably less stable than the finite-difference method, 
in agreement with well-known results \cite{hesthaven2007spectral}.
\begin{figure}[t]
\centerline{\footnotesize
\hspace{-0.0cm}
Second-order centered finite differences
\hspace{3.2cm}
Fourier pseudo-spectral collocation}
\centering
\includegraphics[scale=0.23]{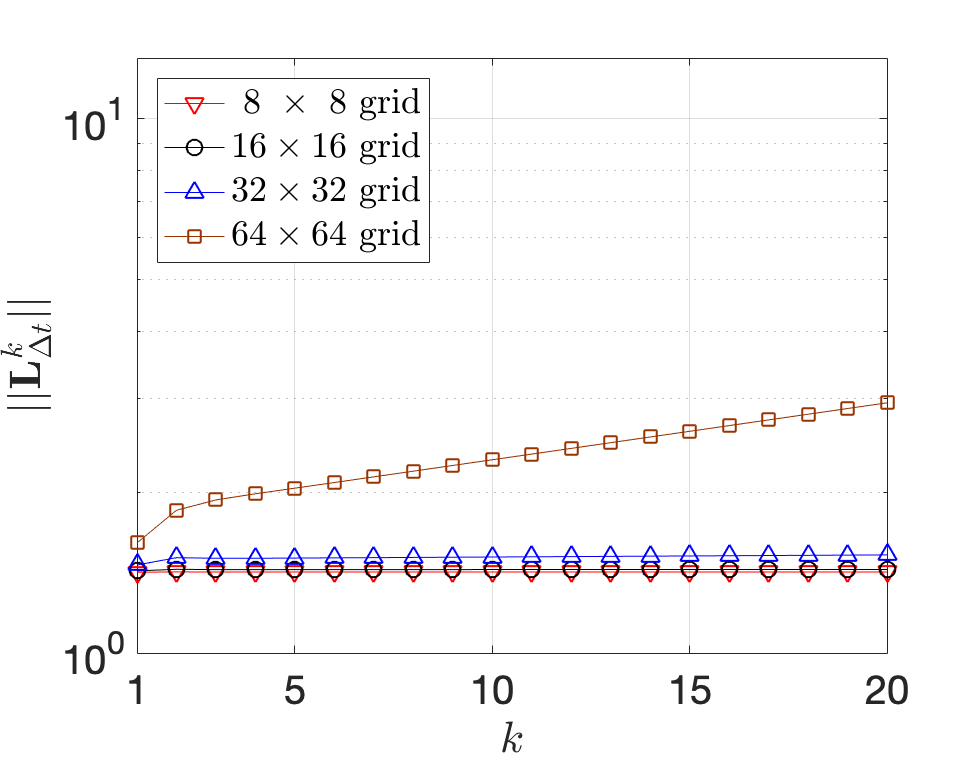}
\includegraphics[scale=0.23]{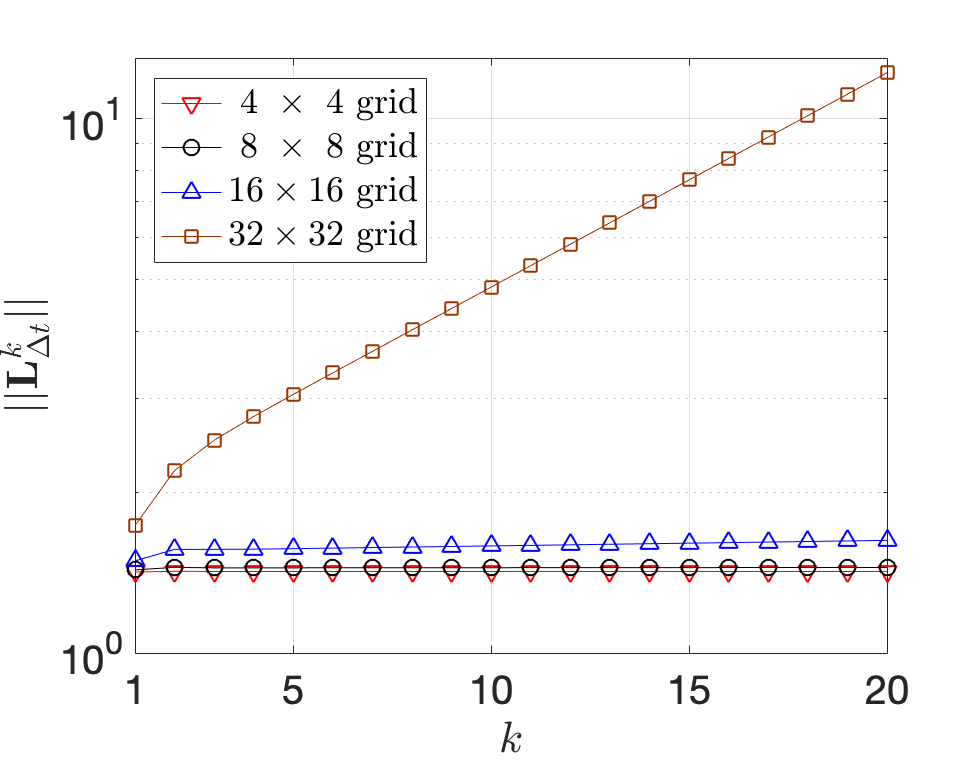}
\caption{
Two-dimensional PDE \eqref{2d-var-coef-advect}. 
Operator norm  of ${\bf L}^k_{\Delta t}$ (see Eq. \eqref{L2d})
versus $k$ for two conditionally 
stable schemes, namely the second-order centered 
finite-differences and the Fourier pseudo-spectral   
collocation schemes on a grid with $n\times n$ evenly-spaced 
points in $[0,2\pi]^2$, with  $n=4, 8, 16, 32, 64$.
It is seen that that the Fourier pseudo-spectral method 
is less stable than the finite-difference method, 
in agreement with well-known results \cite{hesthaven2007spectral}.
Here we set $\Delta t=0.0025$. }
\label{fig:2d-op-norm}
\end{figure}
\begin{figure}[ht!]
\centering
\centerline{\hspace{1.2cm}
Full-rank solution \hspace {3.2cm}
Low-rank solution}
\centerline{\line(1,0){420}}
{\rotatebox{90}{\hspace{2.4cm}\rotatebox{-90}{
\footnotesize$t=1.2$\hspace{1cm}}}}
\includegraphics[scale=0.179]{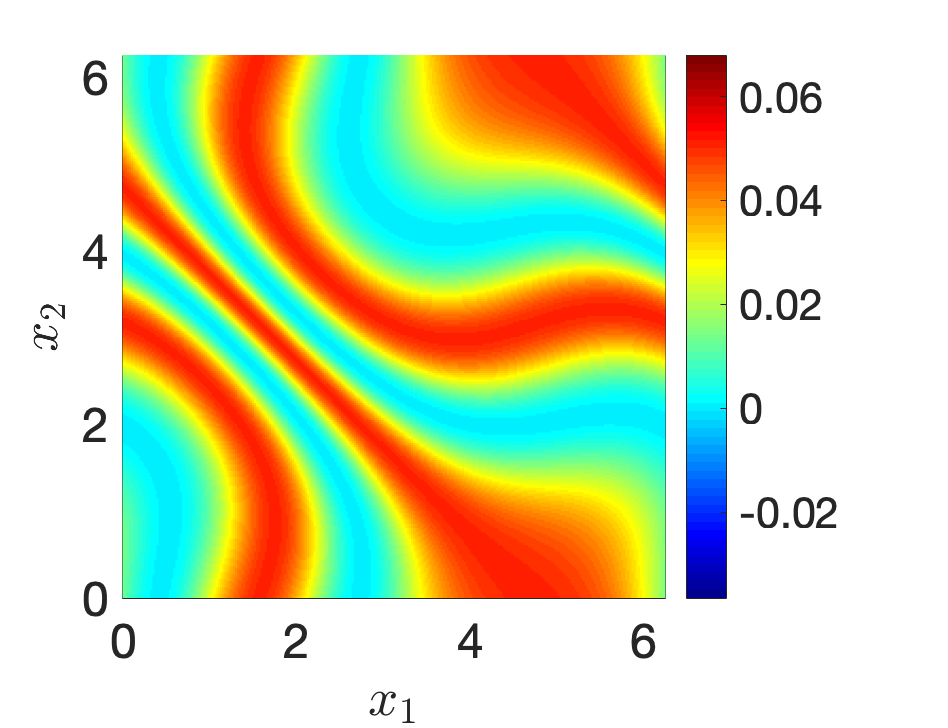}
\includegraphics[scale=0.179]{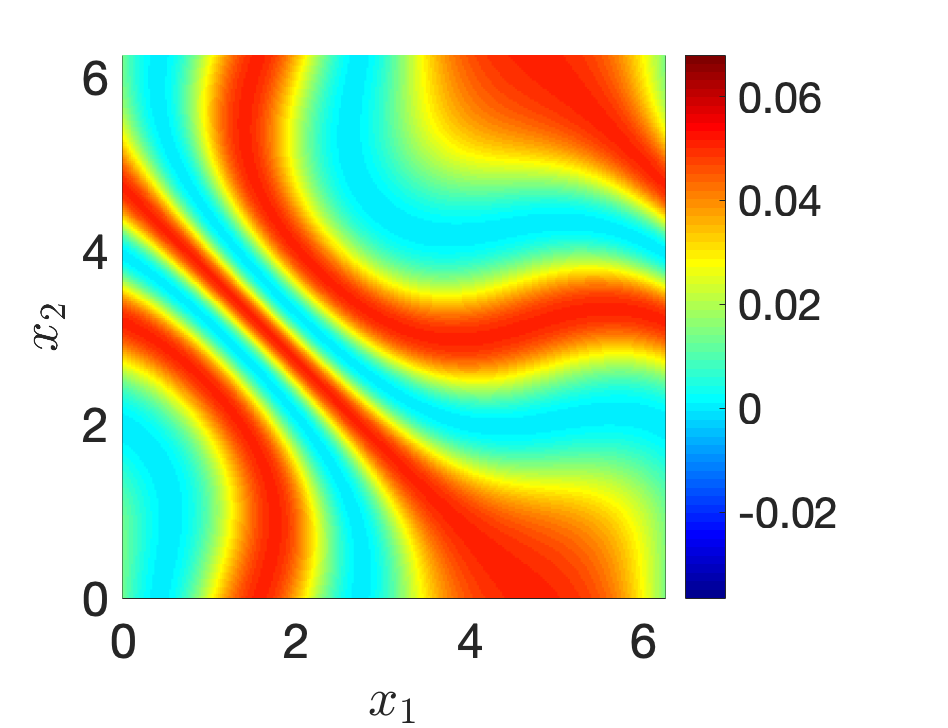}\\
{\rotatebox{90}{\hspace{2.4cm}\rotatebox{-90}{
\footnotesize$t=2.5$\hspace{1cm}}}}
\includegraphics[scale=0.179]{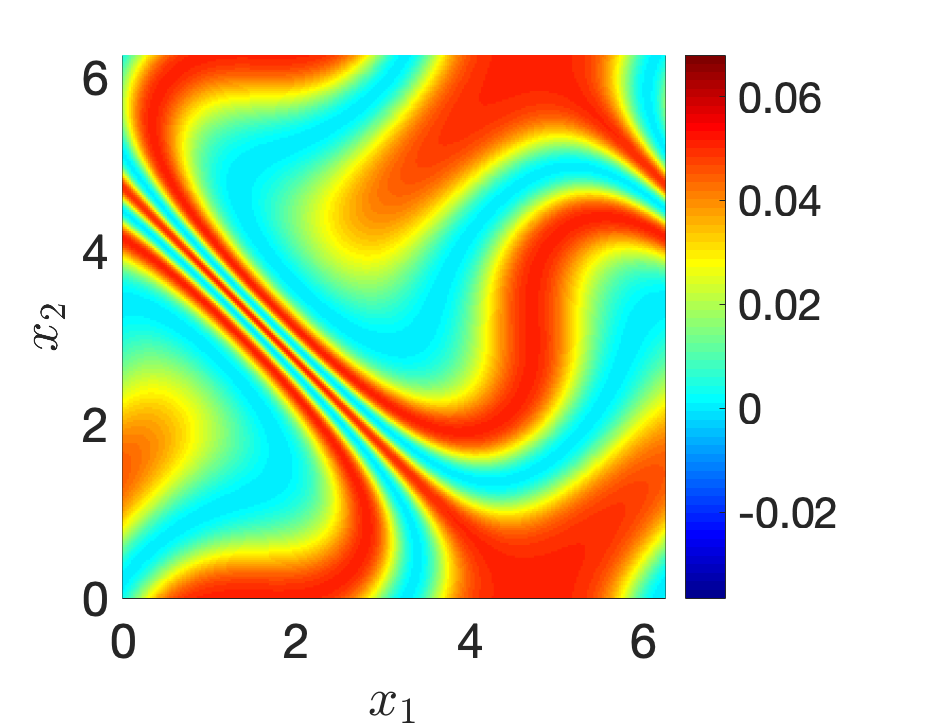}
\includegraphics[scale=0.179]{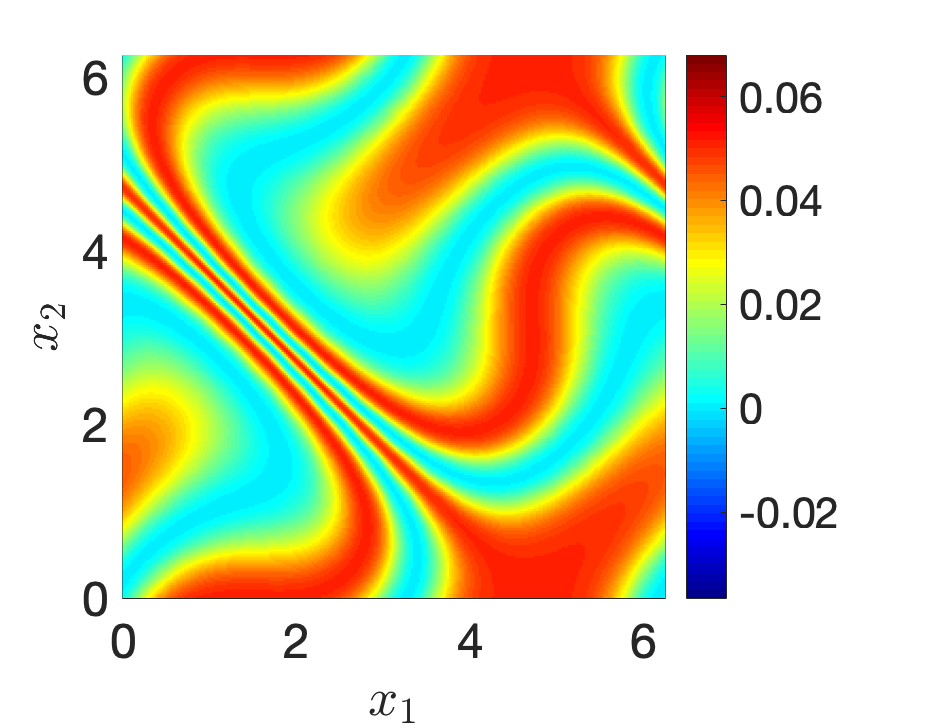}\\
{\rotatebox{90}{\hspace{2.4cm}\rotatebox{-90}{
\footnotesize$t=3.8$\hspace{1cm}}}}
\includegraphics[scale=0.179]{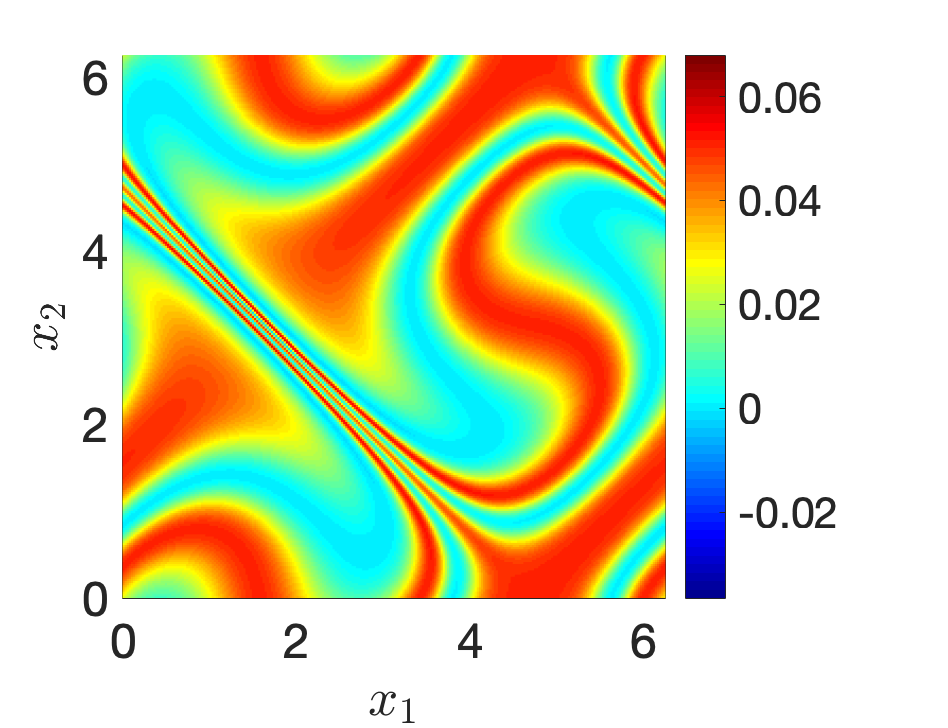}
\includegraphics[scale=0.179]{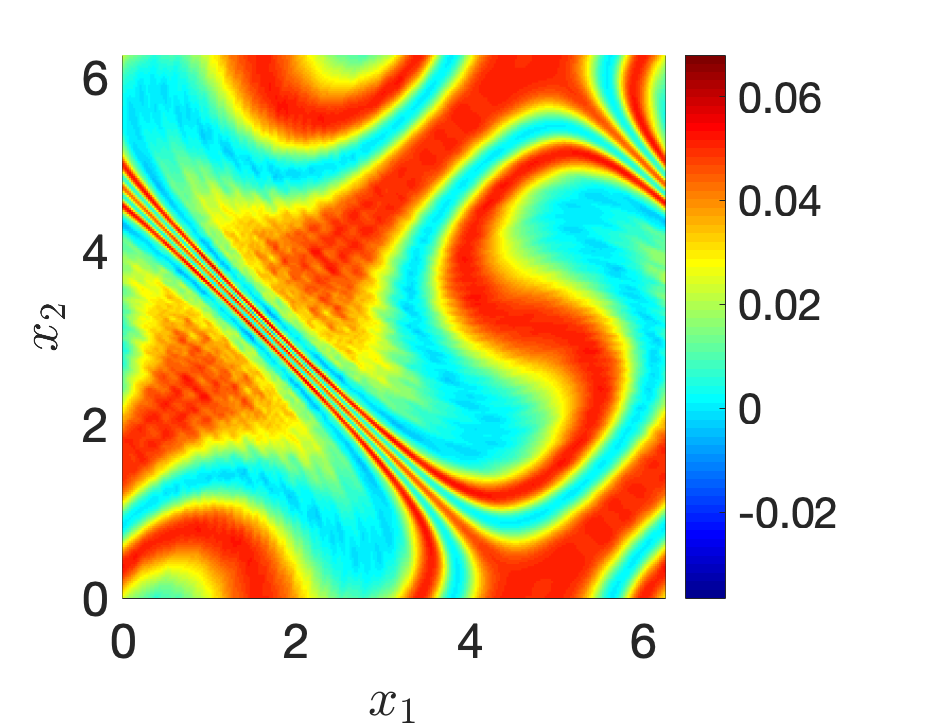}\\
{\rotatebox{90}{\hspace{2.4cm}\rotatebox{-90}{
\footnotesize$t=5.0$\hspace{1cm}}}}
\includegraphics[scale=0.179]{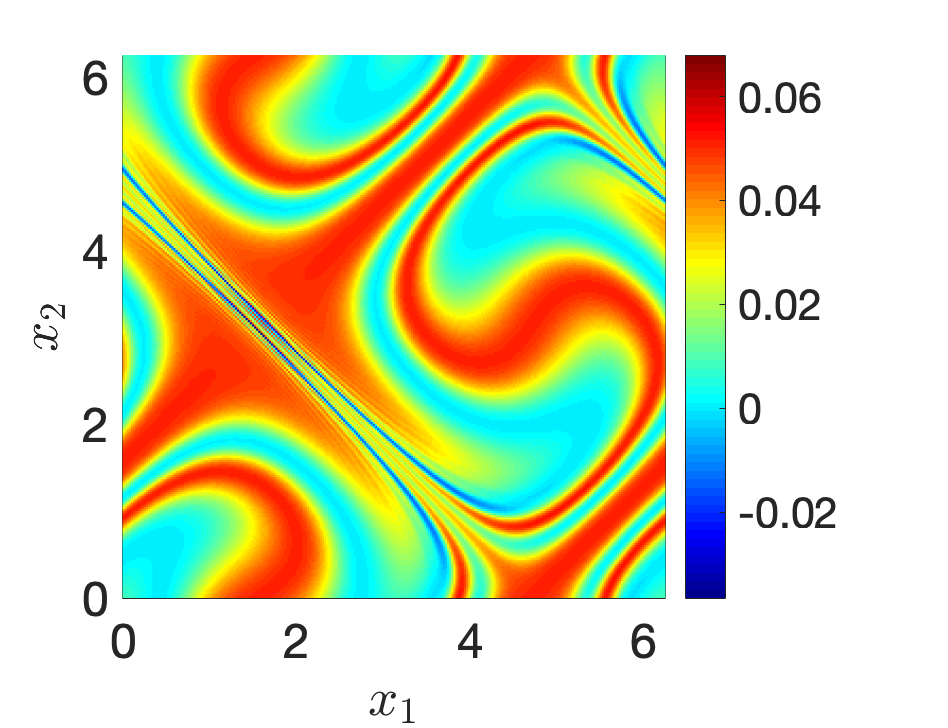}
\includegraphics[scale=0.179]{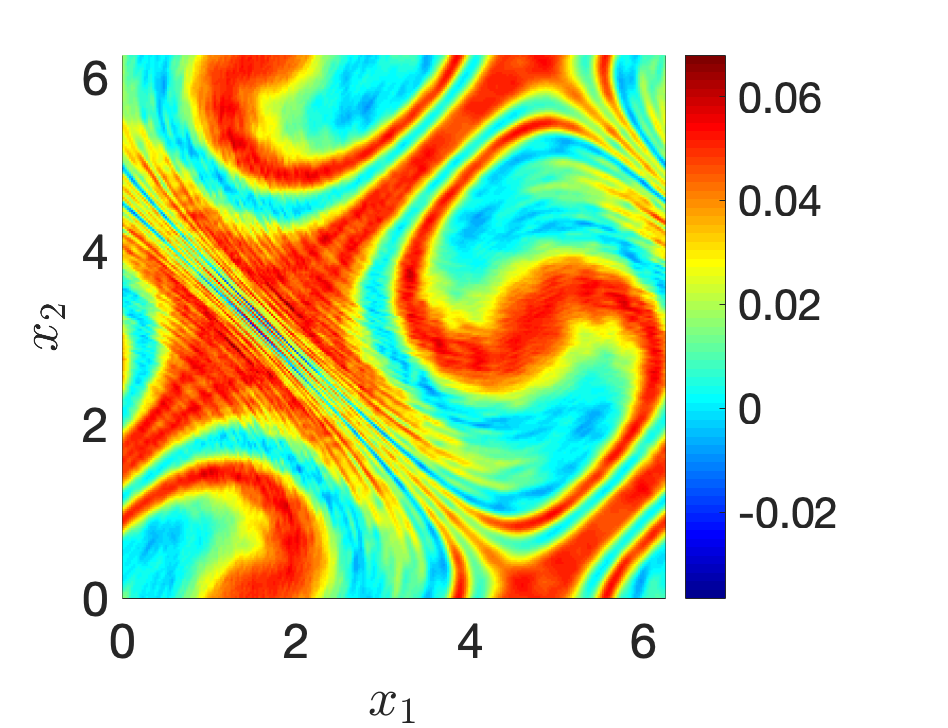}
\caption{
Numerical solution of the PDE \eqref{2d-var-coef-advect} 
using a Fourier pseudo-spectral method on a grid with 
$256\times 256$ nodes. The initial condition is chosen 
as $u_0(x_1,x_2)=\sin^2(x_1+x_2)/(2\pi^2)$.
Shown are the full-rank solution (left) and the low-rank tensor solution 
(right) we obtained by limiting the maximum rank to 64. 
It is seen that the two solutions are 
slightly different at $t=3.8$ and $t=5$, but stability 
is maintained as proven in Theorem \ref{theorem-trunc-bounded}.}
\label{fig:2d-solution-plot}
\end{figure}

In figure \ref{fig:2d-solution-plot} we plot the numerical solution 
of \eqref{2d-var-coef-advect} we obtained with an accurate 
Fourier spectral method. The initial condition is chosen 
as $u_0(x_1,x_2)=\sin^2(x_1+x_2)/(2\pi^2)$. 
It is seen that the low-rank tensor solution obtaining by 
capping the maximum rank to 64 slightly differs from the 
full rank solution at $t=3.8$  and $t=5$. However,  
but stability is maintained as proven in Theorem 
\ref{theorem-trunc-bounded}.
In figure \ref{fig:2d-rank-grow} we show 
that the solution rank grows in time. Such growth 
is determined by the fact that that solution 
to the hyperbolic problem \eqref{2d-var-coef-advect} 
becomes harder to resolve as time increases 
(see figure \ref{fig:2d-solution-plot}). In particular, 
in figure \ref{fig:2d-rank-grow} we see that just 
before applying the truncation operator, the rank of 
the iterate appears to grow at a similar rate to the 
tensor scheme with no truncation.

\begin{figure}[t]
\centerline{\includegraphics[scale=0.25]{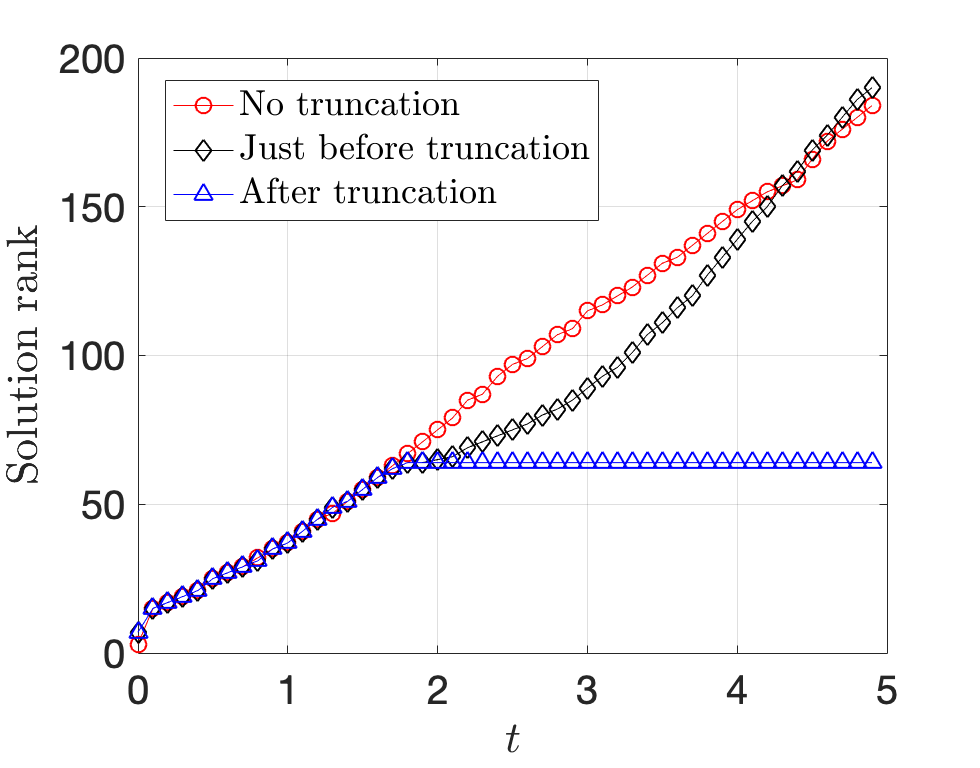}}
\caption{
Tensor rank of the numerical solution to the
PDE \eqref{2d-var-coef-advect} versus time. 
The spatial derivatives are discretized using a 
Fourier pseudo-spectral method on a grid with 
$256\times 256$ nodes. Hence the maximum 
rank of the solution tensor is $256$.  
The rank-limited solution has maximum rank set to 64.
Note that just before applying the truncation operator
in a rank-limited scheme, the rank of the iterate appears to 
grow at a similar rate to the scheme with no truncation. 
The inaccuracies of the rank-truncated solutions shown 
in figure \ref{fig:2d-solution-plot} at $t=3.8$ and $t=5$ 
are due to the fact that the solution rank is much larger
than 64 at such times (compare red and blue curves).}
\label{fig:2d-rank-grow}
\end{figure}

\subsection{Six-dimensional parabolic PDE}

We now demonstrate a higher dimensional example which is both
highly diffusive and very well approximated by
a low-rank numerical solution tensor. To this end, 
we consider the Fokker-Planck equation 
\eqref{fp-eq} with  $d=6$. For our numerical demonstration, 
we consider the following 
drift and diffusion coefficients
\begin{align*}
&{\bf f}({\bf x}) = \begin{bmatrix}
\cos(x_2) & \sin(x_3) & \cos(2x_4) & \sin(2x_5) & \cos(3x_6) & \sin(3x_1)
\end{bmatrix}^T + 6\begin{bmatrix}1 &1 &1 &1 &1 &1\end{bmatrix}^T ,\\
&
{\bf \Gamma}({\bf x})  =\begin{bmatrix}
5\cos^2(x_6) & \sin(x_1) & \cos(x_2) & \sin(x_3) & \cos(x_4) & \sin(x_5)\\
\sin(x_1)&5\cos^2(3x_3) & \sin(5x_2)& \cos(2x_1) & \sin(4x_6)&\cos(x_1)\\
\cos(x_2)&\sin(5x_2)&5\cos^2(3x_5) & \sin(2x_3) & \cos(6x_2) &\sin(x_6)\\
\sin(x_3)&\cos(2x_1)&\sin(2x_3)&5\cos^2(x_3)&\sin(x_1) &\cos(4x_4)\\
\cos(x_4)&\sin(4x_6) & \cos(6x_2) &\sin(x_1) & 5\cos^2(5x_5)& \sin(x_6)\\
\sin(x_5)&\cos(x_1)&\sin(x_6)&\cos(4x_4)& \sin(x_6)&5\cos^2(7x_6) 
\end{bmatrix}+6{\bf I}.
\end{align*}
The matrix of drift coefficients was chosen to encourage mixing between
different variables and the diffusion coefficients were chosen so that the
symmetric matrix is diagonally dominant. Therefore, $\bf \Gamma$ will
always be positive definite ensuring that \eqref{fp-eq} is a 
bounded diffusion problem.
We discretize \eqref{fp-eq} in space using the 
Fourier pseudo-spectral collocation method as 
in \ref{subsec:2d-advect}. Specifically, we construct 
an evenly-spaced grid in $[0,2\pi]^6$ with $31$ points 
in each variable. In principle this yields $31^6$ degrees of 
freedom, which require $110$ MB (Mega Bytes) of storage if 
a tensor product representation in double precision 
floating point arithmetic is utilized. However, 
if we employ a rank $r$ hierarchical Tucker tensor format 
the memory footprint is reduced to 
$[r (31\times 6)+ 4r^3+ r^2]/8$ Bytes. For instance, a 
rank 40 hierarchical Tucker tensor format in 6 dimensions 
on a grid with 31 points in each variable 
requires only $33$ kB (kilo Bytes) of storage.
By using the identity $\partial^2/\partial x_k\partial x_q
=\partial^2/\partial x_q\partial x_k$, we see that we 
can just apply the strict upper triangle part 
of $\bf \Gamma$ once and then double the result. 
The matrix $\bf G$ at right-hand side 
of \ref{mol-ode} in this case has a rather involved 
expression and therefore it is not reported here.
We supplement \eqref{fp-eq} with the initial condition in 
the initial condition 
\begin{equation}
u(0,{\bf x}) = \frac{1}{\pi^6}\prod_{j=1}^6 \sin(x_j)^2.
\label{ic}
\end{equation}
Note that \eqref{ic} is positive and it integrates to one 
over the hyper-cube  $[0,2\pi]^6$, i.e., it is a probability 
density function. 
In figure \ref{fig:time-evolve-marginal} we plot the temporal 
evolution of the marginal PDF 
\begin{equation}
u(t,x_1)=\int_{[0,2\pi]^5} u(t,{\bf x})dx_2\cdots dx_6
\end{equation}
It is seen that the PDE \eqref{fp-eq} is highly diffusive and 
it yields  a solution that can well approximated by a 
low-rank hierarchical Tucker tensor format. 
\begin{figure}[t]
\begin{center}
\centerline{Max rank 1\hspace{6cm} Max rank 5}
\includegraphics[scale=0.22]{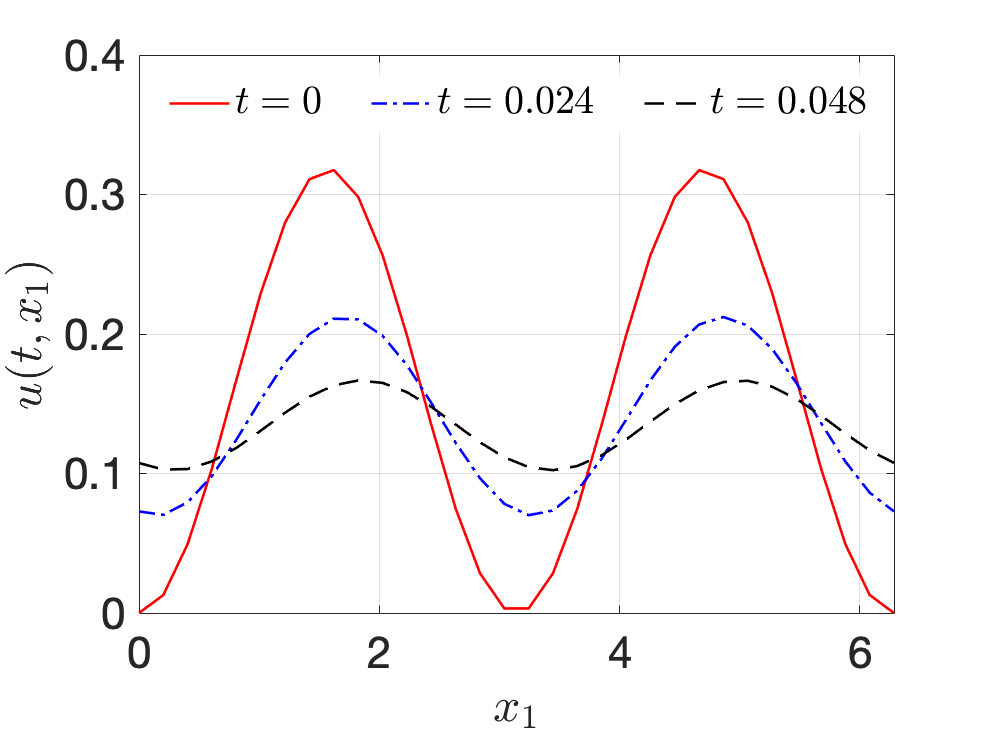}
\includegraphics[scale=0.22]{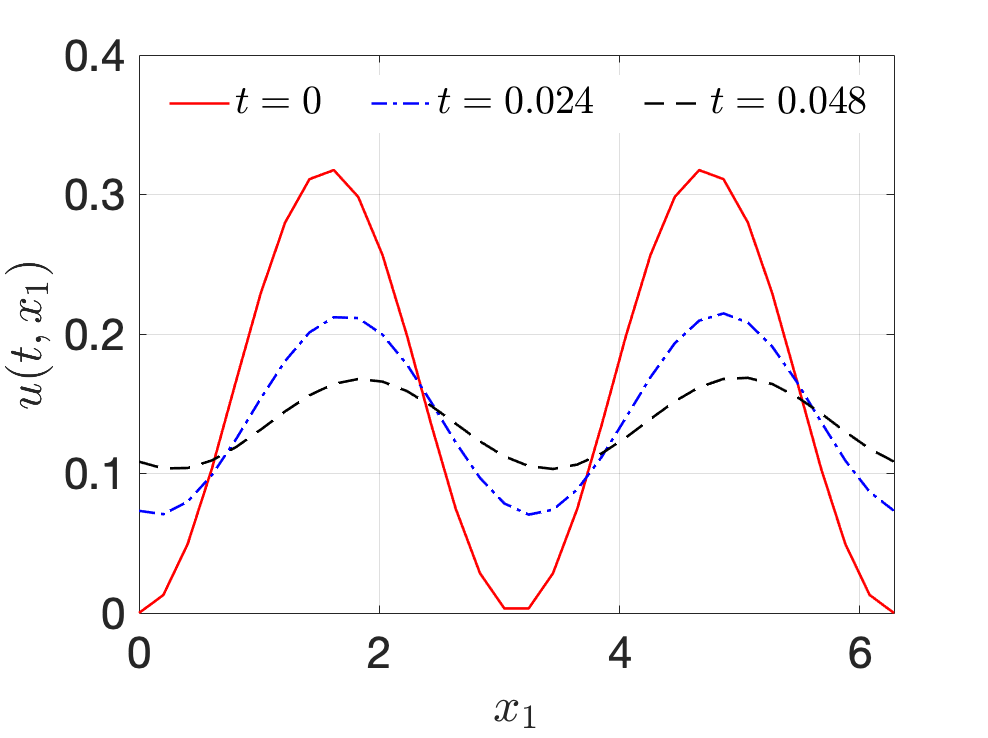}
\caption{Six-dimensional Fokker-Plank equation \eqref{fp-eq}. 
Temporal snapshots of the marginal PDF $u(t,x_1)$. 
It is seen that the system is highly diffusive and it yields  
a solution that can well approximated by a 
low-rank hierarchical Tucker tensor format.}
\label{fig:time-evolve-marginal}
\end{center}
\end{figure}
In figure \ref{fig:lemma_3_1_demo} we provide a 
numerical verification of our Lemma \ref{lemmaTr}. 
To this end, we plot the ratio
$\tau_r(t)=||{\mathfrak T}_{r}({\bf u})||_2/ ||{ \bf u}||_2$
versus time, and verify that is always smaller than one 
for any choice of rank. 
Note that both $\tau_1(t)$ and $\tau_5(t)$ are 
very close to 1, which explains why the two plots in Figure
\ref{fig:time-evolve-marginal} are visually identical.
\begin{figure}[t]
\begin{center}
\includegraphics[scale=0.25]{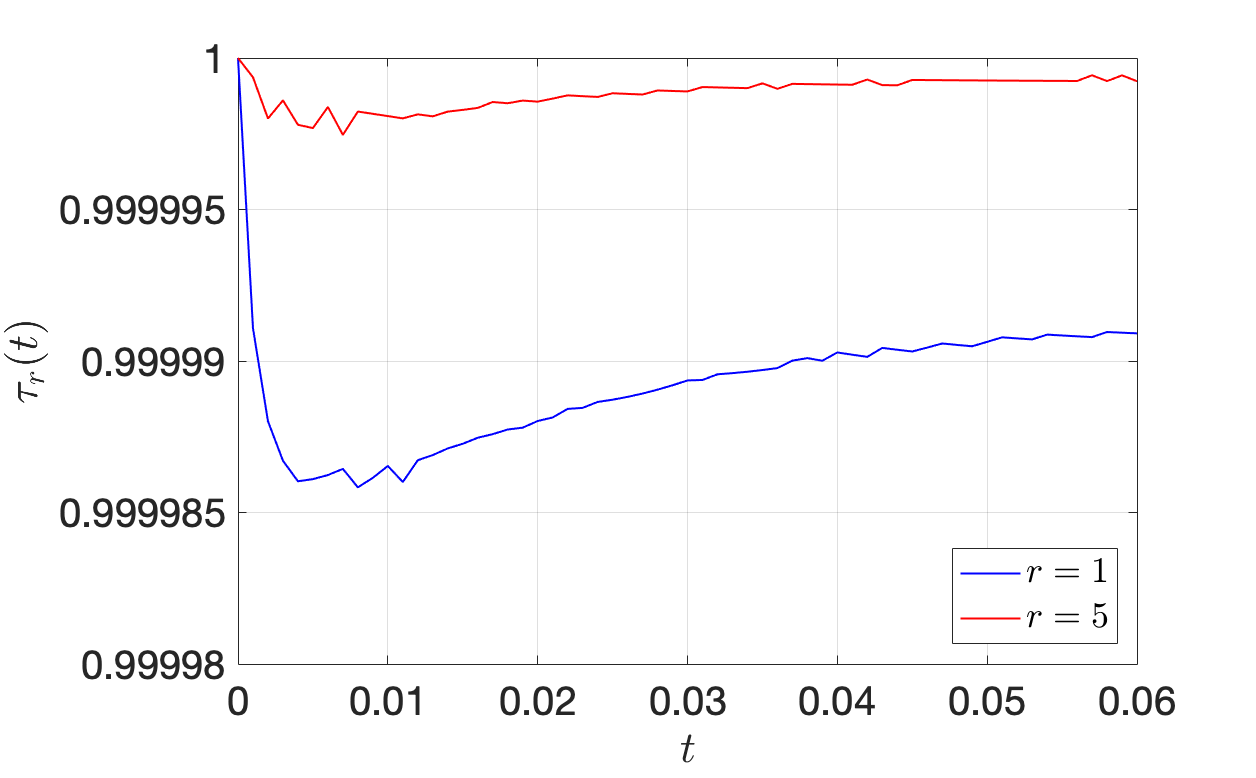}
\caption{
In Lemma \ref{lemmaTr}, we proved that the truncation operator
${\mathfrak T}_{r}$ satisfies the inequality 
$||{\mathfrak T}_{r}({\bf u})||_2 \leq||{ \bf u}||_2$.
In this figure we plot of the ratio 
\mbox{$\tau_r(t)=||{\mathfrak T}_{r}({\bf u})||_2/ ||{ \bf u}||_2$}
versus time. We see that setting max rank equal to 5 does in fact
give us an extra single digit of accuracy in the nonlinear
rank projection. However,$\tau_1(t)$ and $\tau_5(t)$ are 
very close to 1, which explains why the two plots in Figure
\ref{fig:time-evolve-marginal} are visually identical.}
\label{fig:lemma_3_1_demo}
\end{center}
\end{figure}
In our simulations we found that the hierarchical ranks 
of the HT tensor solution do not grow monotonically as 
in the two-dimensional advection problem 
(see figure \ref{fig:2d-rank-grow}). Instead,  they reach a
peak very quickly in time. This is because the solution is 
well approximated by a rank one tensor (see  figure 
\ref{fig:lemma_3_1_demo}).

\section{Summary}
\label{sec:summary}

In this paper we studied stability of linear multistep 
methods (LMM) applied to low-rank tensor discretizations of 
high-dimensional linear PDEs.
In particular, we analyzed the properties of the truncation 
operator the context of iterated maps and proved 
boundedness for a wide range of tensor formats.
This result allowed us to conclude that LMM is stable on 
low-rank tensor manifolds provided that it is stable 
on the full-rank tensor solution. 
We also showed that PDE solvers with explicit 
time-stepping may be subject to severe 
time-step restriction dependent on the 
dimension of the problem, e.g., on the number 
of independent spacial variables.
We provided demonstrative examples of 
truncated linear multistep tensor methods applied 
to variable coefficients linear hyperbolic and 
parabolic PDEs.
Further extensions of the analysis we 
developed in this paper rely on geometric 
integration methods. In particular, it was 
recently shown by Uschmajew and Vandereycken 
\cite{uschmajew2013geometry} that the 
hierarchical Tucker tensor manifold  with 
fixed ranks is smooth. This opens the possibility 
to develop rank-constrained geometric integrators 
\cite{HairerGeom} that preserve the structure 
of the manifold (see, e.g., \cite{Lubich2018,hierar}).

\vspace{0.5cm}
\noindent 
{\bf Acknowledgements} 
This research was supported by the U.S. Army Research Office  
grant W911NF1810309.

\appendix

\section{Brief review of tensor algebra}
\label{subsec:notation}
\noindent
We regard tensors \cite{Kolda} 
as elements of $\mathbb R^{n_1\times \cdots \times n_d}$ . 
Tensors are represented as multidimensional arrays and 
regarded as being multidimensional arrays to the same 
extent that linear operators or bi-linear maps may be 
regarded as matrices, i.e., up to a change in basis. 
A particular entry in a tensor ${\bf A}$ 
is denoted by brackets as ${\bf A}[i_1, \dots , i_d]$ 
where $[i_1, \dots , i_d]\in {\cal I}$ is an array of integers 
called a multi-index. The set $\cal I$ of all multi-indexes 
is called an index set. The tensor product is represented by 
the symbol $\otimes$ and computed for 
${\bf A}\in\mathbb{R}^{\mathcal I}$,
${\bf B}\in\mathbb{R}^{\mathcal J}$ using the definition
\begin{align*}
	({\bf A}\otimes {\bf B})[{i},{j}] =
	{\bf A}[{i}\ ]{\bf B}[{j}]\quad
	\forall{i}\in \mathcal I, \quad \forall {j} \in \mathcal J
\end{align*}
\noindent
Both Kronecker and tensor products result in the same 
array storage in column major format. Their difference 
only lies in what index lengths are specified, i.e., 
$\mathbb{R}^{n_1\cdot n_1}\simeq \mathbb{R}^{n_1\times n_1}$
within the computer when storing an element as an array.

\paragraph{\bf  Matricization of a tensor}
A matricization is specific type of permutation on the components of 
${\bf A}\in \mathbb R^{\mathcal I}$ and its indexes such 
that the resulting tensor is a 2-dimensional array \cite{hackbusch2012tensor}, i.e. a matrix. 
Specifically, let $\cal D$ be an index set, 
$\rho \subseteq {\cal D}$ be an 
ordered subset which will define the rows,  
$\kappa = {\cal D} - \rho $ be an ordered subset 
of all the numbers not in $\rho$, 
${\bm\sigma}_\rho$ be the permutation on sets of 
size $d$ defined by
\begin{align*}
	\bm\sigma_\rho(1,\dots,d) =
		(\rho_1,\dots,\rho_r,\kappa_1,\dots,\kappa_c)
\end{align*}
where $r+c=d$ are the numbers of row indexes and column indexes,
respectively. The $\rho$ mode matricization of $\bf A$ is 
defined by
\begin{align*}
	{\bf A}^{(\rho)}[{\rho},{\kappa} ] = 
	{\bf A}[{i}], \qquad\text{where}\qquad 
	[\rho,\kappa] =  \bm\sigma_\rho(i)
		 \quad \forall
		{i}\in \mathcal I\quad \text{(index set)}.
	\end{align*}
Here $\rho,\kappa$ dermine the row and column of a 
matrix given by applying a row or column major ordering 
scheme to the multi-indexes. Applying the inverse of the 
aforementioned permutation defines a de-matricization, i.e., 
a transformation  back to a tensor. An important case of 
matricization is the vectorization which corresponds to 
listing all the entries of $\bf A$ in a single 
column vector.

\paragraph{\bf $\rho$-mode product}
Let ${\bf A}\in\mathbb R^{\mathcal I}$. Let 
${\bf A}^{(\rho)}$ be a matricization of ${\bf A}$ with 
$R$ rows. Let ${\bf L}\in \mathbb{R}^{M\times R}$ be a matrix.
The $\rho$-mode product between ${\bf L}$ and ${\bf A}$, 
denoted as ${\bf L}\circ_\rho {\bf A}$, is defined 
\cite{grasedyck2010hierarchical} as the tensor satisfying
\begin{align*}
	({\bf L}\circ_\rho {\bf A})^{(\rho)} = {\bf L}{\bf A}^{(\rho)}
\end{align*}
This is the action of multiplying into the $\rho$ index or 
indexes and then summing. The result of the multiplication 
is achieved by applying the de-matricization permutation. 
Also note that if $\otimes$ denotes the
Kronecker product, then 
\begin{align*}
{\bf L}\circ_\lambda (R \circ_{\rho} {\bf A}) \simeq (R \otimes L) {\bf A}^{(\lambda\  \cup\  \rho)}
\label{circ}
\end{align*}
so long as $\lambda\  \cap\  \rho$ is empty. Here,  
$\simeq$ means ``up to matricization permutation''. This property 
is used in Algorithm 941 \cite{kressner2014algorithm} to compute $\rho$-mode products.

\section{Hierarchical Tucker tensor format}
\label{sec:htformat}
\noindent
The Hierarchical Tucker (HT) tensor format is a decomposition of
a tensor obtained by recursively splitting a tensor space into products 
of pairs of spaces along a binary tree \cite{grasedyck2010hierarchical}. 
It was originally introduced by Hackbush and K\"{u}hn in \cite{Hackbusch2009} to 
mitigate the curse of dimensionality and 
storage requirements in the numerical representation 
of the solution to high-dimensional problems Hereafter we 
give a brief overview of the HT tensor format. 
The interested reader is referred to \cite{grasedyck2010hierarchical,grasedyck2018distributed,hackbusch2012tensor,kressner2014algorithm} 
 
 \paragraph{\bf Dimension Tree}
A dimension tree $\mathcal T_d$ with $d\in\mathbb N$ is a tree with an array
of integers associated with each node. The root node is defined as the node
with the array $[1,\dots,d]$. If $t\in \mathcal T_d$ is a node on 
the dimension tree, then the children of $t$ must by definition have 
arrays which can be concatencated to form the array at $t$. 
If a tree node has an array with more than one element, it must 
have children. A node with a singleton array is called a leaf. 
If a node is not a leaf and not the root, it is called interior.
Our definition of dimension tree a slight modification of the definition given in  
\cite{grasedyck2010hierarchical,uschmajew2013geometry,da2015optimization}. 
In particular, we allow for non-binary dimension trees such 
as that of the Tucker format \cite{hackbusch2012tensor,grasedyck2010hierarchical},
which looks like a star network topology if the definition above is applied.
Of course, a dimension tree is binary if each non-leaf node has two 
children called ``left'' and ``right''. This can always be accomplished 
by bisecting the array at a given node. If the array has an odd number 
of elements, give the left child one more than the right.
The HT tensor format corresponds to a binary dimension tree with
a matrix (2-tensor) at its root, 3-tensors at the interior nodes, and
matrices at the leaves. The 3-tensors are called transfer tensors.
If the columns of a leaf matrix are a independent, said matrix is
called a leaf frame.

\paragraph{\bf Hierarchical size of a dimension tree}
Let $\mathcal T_d$ be a dimension tree. A hierarchical size associated with
$\mathcal T_d$ is a mapping from the nodes to the 
natural numbers $\mathbb N$. We denote the size at 
node $t\in\mathcal T_d$ by $r_t$.
The definition of hierarchical size is meant to express number of entries
stored for each multidimensional array. It corresponds to a specific notion of
rank of certain matricizations when finding an estimate of a particular
${\bf A}\in {\mathbb R}^{\cal I}$ given a dimension tree. 
Greater detail is given in \cite{grasedyck2010hierarchical}.

\paragraph{\bf Memory storage format of an HT tensor}
Let $\mathcal I = \mathcal{I}_1\times \dots \times \mathcal{I}_d$
be an index set with dimension $d$ and let
$N_\mu=\max(\mathcal I_\mu)$. Let $\mathcal T_d$ be a 
binary dimension tree with hierarchical sizes $r_t$. 
Let $\rho(t)$ denote the right child of $t$ and let 
$\lambda(t)$ denote the left child. The sizes of the 
tensors on $\mathcal T_d$ in the HT format are:
    \begin{enumerate}
        \item At the root node, $\mathcal D = \{1,\dots,d\}$,
        the matrix is denoted by ${\bf B}_{\mathcal D}$ and its size is
        $r_{\lambda(\mathcal{D})} \times r_{\rho(\mathcal{D})}$.
        We think of ${\bf B}_{\mathcal D}$ as being a bilinear form
        from $\mathbb{R}^{r_{\lambda(\mathcal{D})}}\times 
        \mathbb{R}^{r_{\rho(\mathcal{D})}}$ to $\mathbb R$.
        
        \item In the interior nodes, $t\in\mathcal T_d$, the
        transfer tensors are written as ${\bf B}_t$ and their sizes
        are $r_{\lambda(t)} \times r_{\rho(t)}\times r_{t}$.
        We think of ${\bf B}_t$ as a bilinear map from
        $\mathbb R^{r_{\lambda(t)}} \times \mathbb R^{r_{\rho(t)}}$
        to $\mathbb R^{ r_{t}}$.
        
        \item At the leaves, $\{\mu\}\in\mathcal T_d$, the matrices
        are written as $U_\mu$ and their sizes are
        $N_\mu \times r_{\{\mu\}}$.
    \end{enumerate}

\noindent
Note that the hierarchical sizes are not necessarily 
ranks of any of the tensors defined here.
Further detail regarding the multilinear algebra of
the HT format, including change of basis rules,
is described in \cite{uschmajew2013geometry}.

\paragraph{\bf Hierarchical rank}
Let ${\bf A}\in \mathbb{R}^{\cal I}$ be a tensor corresponding to 
the dimension tree ${\cal T}_d$. The hierarchical rank 
of ${\bf A}$ is the set of hierarchical sizes defined 
by $r_t = \text{rank}({\bf A}^{(t)})$ for every non-root 
$t\in{\cal T}_d$. For the root, we say the
hierarchical rank $r_{t_{root}} = 1$ since the matrix there
can be seen as a $r_{\lambda}\times r_{\rho}\times 1$ tensor.
It is shown in \cite{grasedyck2010hierarchical} that one can
always express a tensor in the HT format using hierarchical
sizes equal to the hierarchical ranks of $\bf A$. 

\paragraph{\bf Hierarchical  truncation}
Hierarchical truncation is a generalization of the
notion of low-rank SVD-base matrix approximation.
This operation is one of the core topics of low-rank tensor 
approximations \cite{grasedyck2010hierarchical}. 
Let ${\bf A}\in \mathbb{R}^{\cal I}$ be a tensor corresponding to 
the dimension tree ${\cal T}_d$, and 
let ${\cal T}_d^l$ denote all nodes which are in
layer $l$ of the tree, i.e., the number of branch
traversals it takes to reach the root. 
The truncation of $\bf A$ is defined as
\begin{align*}
        {\mathfrak T}_r({\bf A}) &=
        \prod_{t\in{\cal T}_d^p}{\bf P}_t
        \cdots
        \prod_{t\in{\cal T}_d^1}{\bf P}_t{\bf A}
    \end{align*}
    where every ${\bf P}_t$ is an orthogonal projection
    formed using $t$-mode matricizations of $\bf A$.
    If $\bf A$ is in the HT format, then the
    resulting truncated tensor may also be expressed
    in the HT format with hierarchical sizes
    defined by the matrix ranks of all ${\bf P}_t$.

\subsection{Parallel implementation}
   Our complete implementation of the HTucker format in parallel is
    found in \cite{rodgers2019htuckermpi}.
    Each tree node is associated with a different compute node
    in a manner similar to \cite{grasedyck2018distributed}.
    Each processor also stores integers to indicate if it is, root,
    leaf, or interior as well as what its parent/children are.
    All nodes the Hierarchical Tucker tensor are stored as C++
    objects which are instantiated in parallel. Each instance
    of an HTucker node object communicates with the other nodes
    on the tree through the Open MPI message passing library.
    Matrices and tensors are passed through the use of a 
    memory format message encoded in long integers
    followed by the components of the array passed as double
    precision floating point numbers. Coding with the library
    in a driver file largely behaves the same as coding in serial,
    with memory handled by the package in parallel. Computation
    on each node is split into cases, where a core is told what to
    do based on if it is root, interior, or a leaf.
    This allows for many algorithms to be implemented easily
    in HTucker format, since tree traversal algorithms can be avoided
    by simply telling the cores to behave in one of three cases.

\bibliography{bram-refs}
\bibliographystyle{plain}

\end{document}